\documentclass[reqno,12pt]{amsart}

\usepackage[utf8]{inputenc}
\usepackage{microtype}
\usepackage{amsfonts}
\usepackage{amsmath}
\usepackage{graphicx}
\usepackage{hyperref}
\usepackage{color}

\usepackage{amssymb}
\usepackage{enumitem}

\usepackage{tikz}
\usetikzlibrary{topaths}

\usepackage{a4wide}

\usepackage{empheq}

\newtheorem{theorem}{Theorem}
\newtheorem{lemma}[theorem]{Lemma}
\newtheorem{proposition}[theorem]{Proposition}

\newcommand{\N}{\mathbb{N}}

\newcommand{\R}{\mathbb{R}}

\newcommand{\lk}{\operatorname{lk}}

\newcommand{\dlk}{\lk^\downarrow\!}

\newcommand{\groupoid}{\mathcal{F}}
\newcommand{\match}{\mathcal{M}}
\newcommand{\tree}{\mathfrak{t}}
\newcommand{\alttree}{\mathfrak{u}}
\newcommand{\altalttree}{\mathfrak{v}}
\newcommand{\forest}{\mathfrak{f}}
\newcommand{\altforest}{\mathfrak{e}}

\newcommand{\defeq}{\mathbin{\vcentcolon =}}

\DeclareMathOperator{\F}{F}

\DeclareMathOperator{\CAT}{CAT}

\DeclareMathOperator{\core}{core}

\numberwithin{equation}{section}

\begin{document}

\title[A short account of why Thompson's group $F$ is of type $\F_\infty$]{A short account of why Thompson's group $F$\\ is of type $\F_\infty$}

\author[M.~C.~B.~Zaremsky]{Matthew C.~B.~Zaremsky}
\address{Department of Mathematics and Statistics, University at Albany (SUNY), Albany, NY 12222}
\email{mzaremsky@albany.edu}

\begin{abstract}
In 1984 Brown and Geoghegan proved that Thompson's group $F$ is of type $\F_\infty$, making it the first example of an infinite dimensional torsion-free group of type $\F_\infty$. Over the decades a different, shorter proof has emerged, which is more streamlined and generalizable to other groups. It is difficult, however, to isolate this proof in the literature just for $F$ itself, with no complicated generalizations considered and no additional properties proved. The goal of this expository note then is to present the ``modern'' proof that $F$ is of type $\F_\infty$, and nothing else.
\end{abstract}

\maketitle
\thispagestyle{empty}

\subsection*{Introduction and History}

A \emph{classifying space} for a group $G$ is a CW complex $Y$ with $\pi_1(Y)\cong G$ and $\pi_k(Y)=0$ for all $k\ne 1$. If $G$ admits a classifying space with finite $n$-skeleton, we say $G$ is of \emph{type $F_n$}. Equivalently, $G$ is of type $\F_n$ if it admits a free, cocompact, cellular action on an $(n-1)$-connected CW complex. Being of type $\F_1$ is equivalent to being finitely generated, and being of type $\F_2$ is equivalent to being finitely presented. We say $G$ is of \emph{type $F_\infty$} if it is of type $\F_n$ for all $n$. Thompson's group $F$ was the first example of a torsion-free group of type $\F_\infty$ with no finite dimensional classifying space. Indeed, $F$ cannot have a finite dimensional classifying space since it turns out to contain infinite rank free abelian subgroups, and so in some sense is an ``infinite dimensional'' group.

The fastest definition of $F$ is via the infinite presentation
\[
F=\langle x_0,x_1,x_2,\dots\mid x_j x_i=x_i x_{j+1} \text{ for all } i<j\rangle\text{.}
\]
The endomorphism $F\to F$ sending each $x_i$ to $x_{i+1}$ is clearly idempotent up to conjugation, and so encodes a ``free homotopy idempotent''. This is another source of the group's fame, as it turns out $F$ is the universal group encoding a free homotopy idempotent. This was proved by Freyd--Heller \cite{freyd93} (written in the 1970s, published in the 1990s) and, independently, Dydak \cite{dydak77}. The notation ``$F$'' originated from Freyd and Heller's work, standing for ``free homotopy idempotent'' (and later sometimes also standing for ``Freyd--Heller group''). The group now called $F$ actually predated Freyd and Heller's work, as it arose in Richard Thompson's work on logic in the 1960s (see, e.g., \cite{mckenzie73}), and so over the years has come to be universally called ``Thompson's group $F$''.

The original proof that $F$ is of type $\F_\infty$ was given by Brown and Geoghegan in \cite{brown84}, making heavy use of the map $F\to F$ that is idempotent up to conjugation. This proof is very specific to $F$ itself, but Brown subsequently found a new proof in \cite{brown87} that generalized more easily to variations of $F$. This proof approach was then simplified and further generalized over the years by Stein \cite{stein92}, Farley \cite{farley03}, and others, in a variety of applications to families of ``Thompson-like'' groups. There are too many examples of this to list here, but lists of such examples can be found in, e.g., \cite{skipper19,witzel19}.

By now a comparatively short, easy proof that $F$ is of type $\F_\infty$ exists, thanks to all this work over the years, but isolating it in the literature is difficult. Many (but not all) of the most important steps can be found in \cite[Section~9.3]{geoghegan08} or \cite{brown92}. Also, one can sort out the full ``modern'' $\F_\infty$ proof for $F$ from the (long) $\F_\infty$ proof for the braided Thompson groups in \cite{bux16}, but this requires quite a bit of effort.

The purpose of this note then is to present the most modern form of the $\F_\infty$ proof for Thompson's group $F$, and only for $F$, with no other groups considered and no other properties proved. The target audience is people interested in understanding the most basic situation, just for $F$, before venturing into more complicated generalizations.

\subsection*{Acknowledgments} Thanks are due to a number of people for encouraging me to write this up, including Brendan Mallery, David Rosenthal, Rachel Skipper, Rob Spahn, and Marco Varisco. I am also grateful to the anonymous referee for many excellent suggestions, which greatly improved the discussion of the history of $F$. This work is supported by grant \#635763 from the Simons Foundation.

\bigskip

\subsection*{A: Trees and forests}

Throughout this note, a \emph{tree} will mean a finite rooted binary tree. A \emph{forest} is a disjoint union of finitely many trees. The roots and leaves of a tree or forest are always ordered left to right. The \emph{trivial tree} is the tree with $1$ leaf (which is also its root). A \emph{trivial forest} is a forest each of whose trees is trivial. We denote the trivial forest with $n$ roots (and hence $n$ leaves) by $1_n$. If we want to avoid specifying $n$, we will just write $1=1_n$.

A \emph{caret} is a tree with $2$ leaves. Given a forest $\forest$, a \emph{simple expansion} of $\forest$ is a forest obtained by adding one new caret to $\forest$, with the root of the caret identified with a leaf of $\forest$. If it is the $k$th leaf, this is the \emph{$k$th simple expansion} of $\forest$ (see Figure~\ref{fig:expand}). An \emph{expansion} of $\forest$ is recursively defined to be $\forest$ or a simple expansion of an expansion of $\forest$. Note that if $\forest$ and $\forest'$ have the same number of roots then (and only then do) they have a common expansion. For example any two trees have a common expansion.

\medskip

\begin{figure}[h]
\begin{tikzpicture}
\draw (0,0) -- (1,1) -- (2,0);
\draw (1,-1) -- (2,0) -- (3,-1);
\filldraw (0,0) circle (1.5pt)   (1,1) circle (1.5pt)   (2,0) circle (1.5pt)   (1,-1) circle (1.5pt)   (3,-1) circle (1.5pt);
\begin{scope}[xshift=3in]
\draw (0,0) -- (1,1) -- (2,0);
\draw (1,-1) -- (2,0) -- (3,-1);
\draw (0,-2) -- (1,-1) -- (2,-2);
\filldraw (0,0) circle (1.5pt)   (1,1) circle (1.5pt)   (2,0) circle (1.5pt)   (1,-1) circle (1.5pt)   (3,-1) circle (1.5pt)   (0,-2) circle (1.5pt)   (2,-2) circle (1.5pt);
\end{scope}
\end{tikzpicture}
\caption{A tree, and a simple expansion of the tree (namely the $2$nd simple expansion).}
\label{fig:expand}
\end{figure}
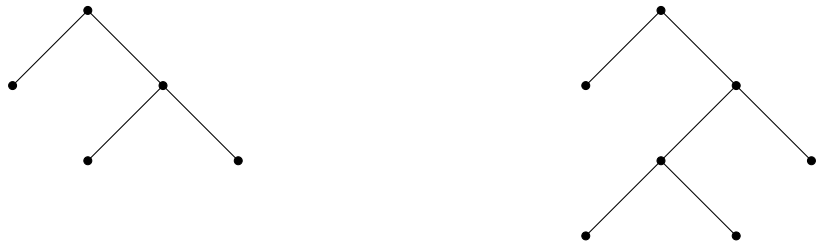

\medskip

\subsection*{B: The group}

A \emph{tree pair} is a pair $(\tree_-,\tree_+)$ where $\tree_\pm$ are trees with the same number of leaves. A \emph{simple expansion} of a tree pair is a tree pair $(\tree_-',\tree_+')$ such that there exists $k$ where $\tree_\pm'$ is the $k$th simple expansion of $\tree_\pm$. An \emph{expansion} of $(\tree_-,\tree_+)$ is recursively defined to be $(\tree_-,\tree_+)$ or a simple expansion of an expansion of $(\tree_-,\tree_+)$. Thompson's group $F$ is the set of equivalence classes $[\tree_-,\tree_+]$ of tree pairs $(\tree_-,\tree_+)$, with the equivalence relation generated by $(\tree_-,\tree_+)\sim (\tree_-',\tree_+')$ whenever $(\tree_-',\tree_+')$ is an expansion of $(\tree_-,\tree_+)$ (for more details on this and some equivalent definitions of $F$, including relating this to the infinite presentation given above, see, e.g., \cite{cannon96,belk04}).

The point of expansions is that one can multiply equivalence classes $[\tree_-,\tree_+]$ and $[\alttree_-,\alttree_+]$ by expanding until without loss of generality $\tree_+=\alttree_-$, and then $[\tree_-,\tree_+][\alttree_-,\alttree_+]\defeq [\tree_-,\alttree_+]$. In this way, $F$ is a group. The identity is $[1_1,1_1]$ and the inverse of an element $[\tree_-,\tree_+]$ is $[\tree_+,\tree_-]$.

\medskip

\subsection*{C: The groupoid}

A \emph{groupoid} is a set with all the axioms of a group except the product $gh$ need not necessarily be defined for every pair of elements $g,h$. A standard example is the set of all square matrices, where two elements can be multiplied if and only if they have the same dimension. Thompson's group $F$ naturally lives in the groupoid where we generalize trees to forests, which we describe now.

A \emph{forest pair} is a pair $(\forest_-,\forest_+)$ where $\forest_\pm$ are forests with the same number of leaves. An \emph{expansion} of a forest pair is defined analogously to an expansion of a tree pair, and we define equivalence of forest pairs similarly to equivalence of tree pairs. Let $\groupoid$ be the set of all equivalence classes $[\forest_-,\forest_+]$ of forest pairs $(\forest_-,\forest_+)$. Since any two forests with the same number of roots have a common expansion, we can multiply two elements $[\forest_-,\forest_+]$ and $[\altforest_-,\altforest_+]$ of $\groupoid$ provided the number of roots of $\forest_+$ and $\altforest_-$ are the same. In this case we expand until $\forest_+=\altforest_-$ and then $[\forest_-,\forest_+][\altforest_-,\altforest_+]\defeq [\forest_-,\altforest_+]$. In this way, $\groupoid$ is a groupoid. Note that the group $F$ is a subgroupoid of $\groupoid$.

\medskip

\subsection*{D: The poset}

Define a \emph{split} to be an element of $\groupoid$ of the form $[\forest,1]$. For $[\forest_-,\forest_+]\in \groupoid$, declare that $[\forest_-,\forest_+]\le [\forest_-,\forest_+][\forest,1]$ for any split $[\forest,1]$ such that this product is defined.

\begin{lemma}\label{lem:poset}
The relation $\le$ is a partial order.
\end{lemma}

\begin{proof}
Clearly $\le$ is reflexive, since any $[1_n,1_n]$ is a split. A product of splits is itself a split, because any forest with $n$ roots is an expansion of $1_n$, so $\le$ is transitive. Finally, a product of non-trivial splits is non-trivial since any expansion of a non-trivial forest is non-trivial, so $\le$ is antisymmetric.
\end{proof}

Let $\groupoid_1$ be the subset of $\groupoid$ consisting of all $[\tree,\forest]$ for $\tree$ a tree (and $\forest$ a forest with the same number of leaves as $\tree$). The groupoid product on $\groupoid$ restricts to a left action of $F$ on $\groupoid_1$. It is clear that $\le$ restricts to $\groupoid_1$, and that this partial order on $\groupoid_1$ is $F$-invariant, since left multiplication by an element of $F$ commutes with right multiplication by a split. In this way, $\groupoid_1$ is an $F$-poset.

The \emph{geometric realization} $|P|$ of a poset $P$ is the simplicial complex with a simplex for every chain $x_0<\cdots<x_k$ of elements $x_i\in P$, with face relation given by taking subchains. A poset is \emph{directed} if any two elements have a common upper bound. It is a standard fact that the geometric realization of a directed poset is contractible.

\begin{lemma}
The poset $\groupoid_1$ is directed, and so the geometric realization $|\groupoid_1|$ is contractible.
\end{lemma}

\begin{proof}
Note that $[\tree,\forest][\forest,1]=[\tree,1]$, so any element of $\groupoid_1$ has an upper bound of the form $[\tree,1]$ for $\tree$ a tree. Given two such elements $[\tree,1]$ and $[\alttree,1]$, let $\altalttree$ be a common expansion of $\tree$ and $\alttree$, and now $[\altalttree,1]$ is a common upper bound of $[\tree,1]$ and $[\alttree,1]$.
\end{proof}

Since the action of $F$ on $\groupoid_1$ is order preserving, it induces a simplicial action of $F$ on the contractible complex $|\groupoid_1|$.

\begin{lemma}\label{lem:free}
The action of $F$ on $|\groupoid_1|$ is free.
\end{lemma}

\begin{proof}
The action of $F$ on $|\groupoid_1|^{(0)}=\groupoid_1$ is free, since it is an action of a subgroup of a groupoid on the groupoid by left translation. Since the action of $F$ on $\groupoid_1$ is order preserving, the stabilizer of the simplex $x_0<\cdots<x_k$ lies in the stabilizer of $x_0$, hence is trivial.
\end{proof}

\medskip

\subsection*{E: The Stein complex}

In \cite{brown87} Brown used the action of $F$ on $|\groupoid_1|$ to give a new proof that $F$ is of type $\F_\infty$, which generalized to many additional groups. The topological analysis in \cite{brown87} was still quite complicated though. The complex $|\groupoid_1|$ deformation retracts to a smaller, more manageable subcomplex $X$ now called the \emph{Stein complex}. This complex was first constructed by Stein in \cite{stein92} (also see \cite{brown92}), and simplified the $\F_\infty$ proof for $F$ in \cite{brown87} quite a bit.

To define $X$ we need the notion of ``elementary'' forests, splits, and simplices. First, call a forest $\forest$ \emph{elementary} if every tree in $\forest$ is either trivial or a single caret (see Figure~\ref{fig:elem}). Call a split $[\forest,1]$ \emph{elementary} if $\forest$ is an elementary forest. If $x\in \groupoid_1$ and $s$ is a split, so $x\le xs$, then write $x\preceq xs$ if $s$ is an elementary split. (Note that $\preceq$ is reflexive and antisymmetric, but not transitive.) Call a simplex $x_0<\cdots<x_k$ in $|\groupoid_1|$ \emph{elementary} if $x_i\preceq x_j$ for all $i<j$. The elementary simplices form a subcomplex $X$, called the \emph{Stein complex}. Note that $X$ is invariant under the action of $F$.

\medskip

\begin{figure}[h]
\begin{tikzpicture}
\draw (0,0) -- (1,1) -- (2,0);
\filldraw (0,0) circle (1.5pt)   (1,1) circle (1.5pt)   (2,0) circle (1.5pt)   (-1,1) circle (1.5pt);
\begin{scope}[xshift=3in]
\draw (0,0) -- (1,1) -- (2,0);
\draw (1,-1) -- (2,0) -- (3,-1);
\filldraw (0,0) circle (1.5pt)   (1,1) circle (1.5pt)   (2,0) circle (1.5pt)   (-1,1) circle (1.5pt)   (1,-1) circle (1.5pt)   (3,-1) circle (1.5pt);
\end{scope}
\end{tikzpicture}
\caption{An example of an elementary forest and a non-elementary forest.}
\label{fig:elem}
\end{figure}
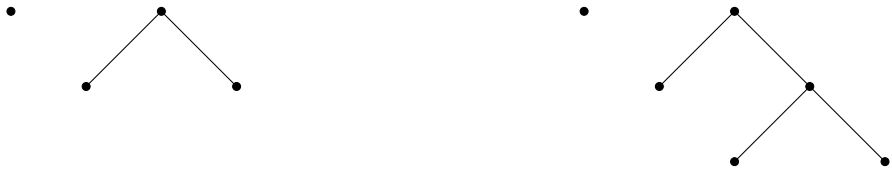

\begin{proposition}\label{prop:stein_cible}
The Stein complex $X$ is homotopy equivalent to $|\groupoid_1|$, hence is contractible.
\end{proposition}

\begin{proof}
Given a forest $\forest$, there is a unique maximal elementary forest with $\forest$ as an expansion, namely the elementary forest whose $k$th tree is non-trivial (hence a caret) if and only if the $k$th tree of $\forest$ is non-trivial, for each $k$. Call this the \emph{elementary core} of $\forest$, denoted $\core(\forest)$. Note that if $\forest$ is non-trivial then so is $\core(\forest)$. If $\altforest=\core(\forest)$, call $[\altforest,1]$ the \emph{elementary core} of $[\forest,1]$, and write $\core([\forest,1])\defeq [\altforest,1]$. Now let $x\le z$ with $x\not\preceq z$, and consider $(x,z)\defeq \{y\mid x<y<z\}$. Since any $y\in (x,z)$ is of the form $xs$ for $s$ a non-trivial split, we can define a map $\phi\colon (x,z)\to (x,z)$ via $\phi(xs)\defeq x\core(s)$. This is clearly a poset map that restricts to the identity on its image, and satisfies $\phi(y)\le y$ for all $y$. Finally, note that $\phi(y)\le\phi(z)\in (x,z)$ for all $y$. Standard poset theory (see, e.g., \cite[Section~1.5]{quillen78}) now tells us that $|(x,z)|$ is contractible (intuitively, $\phi$ ``retracts'' it to a cone on the point $\phi(z)$).

Now our goal is to build up from $X$ to $|\groupoid_1|$ by gluing in the missing simplices, in such a way that whenever we add a new simplex it is along a contractible relative link, which will imply that $X\simeq |\groupoid_1|$. The missing simplices are precisely the non-elementary ones. Let us actually glue in all the non-elementary simplices in chunks, by gluing in (contractible) subcomplexes of the form $|\{y\mid x\le y\le z\}|$ for $x<z$ non-elementary. We glue these in, in order of increasing $f(z)-f(z)$ value, where $f\colon \groupoid_1\to \N$ sends $[\tree,\forest]$ to the number of roots of $\forest$. (Think of $f(z)-f(x)$ as the number of carets in the split taking $x$ to $z$.) When we glue in $|\{y\mid x\le y\le z\}|$, the relative link is $|\{y\mid x\le y<z\}|\cup|\{y\mid x<y\le z\}|$. This is the suspension of $|\{y\mid x<y<z\}|$, which is contractible by the first paragraph of the proof.
\end{proof}

Note that the action of $F$ on $|\groupoid_1|$ restricts to an action of $F$ on $X$.

\medskip

\subsection*{F: The Stein--Farley cube complex}

The Stein complex $X$ is easier to use than $|\groupoid_1|$, but there is one further simplification that makes it still easier, namely, the simplices of $X$ can be glommed together into cubes, making $X$ a cube complex. This was observed by Stein in \cite{stein92} and further developed in \cite{farley03}, where $X$ was shown to even be a $\CAT(0)$ cube complex.

Given $x\preceq z$, say $z=xs$ for $s=[\altforest,1]$ an elementary split, the set $\{y\mid x\le y\le z\}$ is a boolean lattice. This is because the forests $\altforest'$ with $x[\altforest',1]\le x[\altforest,1]$ are all obtained by assigning a $0$ or a $1$ to each caret in $\altforest$ and including said caret in $\altforest'$ if and only if it was assigned a $1$. The geometric realizations of these boolean lattices, which are metric cubes, cover $X$, and any non-empty intersection of such cubes is itself such a cube, so in this way $X$ has the structure of a cube complex. When we view $X$ as a cube complex instead of a simplicial complex, we will call it the \emph{Stein-Farley complex}.

The action of $F$ on $X$ takes cubes to cubes, so $F$ acts cellularly on the Stein-Farley complex.

\medskip

\subsection*{G: Sublevel complexes}

At this point we have a free cellular action of $F$ on the contractible cube complex $X$. If the orbit space were of finite type, meaning if it had finitely many cells per dimension, then we would be done proving $F$ is of type $\F_\infty$. In fact this is not the case, but $X$ does admit a natural filtration into cocompact subcomplexes that are increasingly highly connected, as we now explain.

Let $f\colon \groupoid_1\to \N$ be the function from the proof of Proposition~\ref{prop:stein_cible}, so $f([\tree,\forest])$ equals the number of roots of $\forest$. Note that $f$ is $F$-invariant. For each $m\in\N$ let $X^{f\le m}$ be the full subcomplex of $X$ spanned by vertices $x\in X^{(0)}=\groupoid_1$ with $f(x)\le m$. The $X^{f\le m}$ are called \emph{sublevel complexes}. Note that the $X^{f\le m}$ are nested, and their union is all of $X$, so they form a filtration of $X$. Each $X^{f\le m}$ is $F$-invariant.

\begin{lemma}\label{lem:cocpt}
Each $X^{f\le m}$ is cocompact under the action of $F$.
\end{lemma}

\begin{proof}
Since there are only finitely many elementary forests with a given number of roots or a given number of leaves, $X$ is locally finite. Hence it suffices to show $X^{f\le m}$ has finitely many $F$-orbits of vertices, and for this we claim that if $x,x'\in X^{(0)}$ with $f(x)=f(x')$ then $F.x=F.x'$. Indeed, $f(x)=f(x')$ ensures that $x'x^{-1}$ is an allowable product in $\groupoid$, and clearly $x'x^{-1}\in F$ with $(x'x^{-1})x=x'$.
\end{proof}

To summarize, for each $m\in\N$, $F$ acts freely, cocompactly, and cellularly on $X^{f\le m}$. To show that $F$ is of type $\F_\infty$, i.e., of type $\F_n$ for all $n$, it just remains to show that for each $n$ there exists $m$ such that $X^{f\le m}$ is $(n-1)$-connected.

Let $\displaystyle \nu(m)\defeq \left\lfloor\frac{m-2}{3}\right\rfloor$.

\begin{proposition}\label{prop:conn}
The complex $X^{f\le m}$ is $(\nu(m+1)-1)$-connected.
\end{proposition}

We will prove Proposition~\ref{prop:conn} shortly. First let us see why we will be done after this.

\begin{theorem}
$F$ is of type $\F_\infty$.
\end{theorem}

\begin{proof}
For each $m\in\N$, $F$ acts freely, cocompactly, and cellularly on the $(\nu(m+1)-1)$-connected complex $X^{f\le m}$. Hence $F$ is of type $\F_{\nu(m+1)}$ for all $m$. Since $\nu(m+1)$ goes to $\infty$ as $m$ goes to $\infty$, $F$ is of type $\F_\infty$.
\end{proof}

\medskip

\subsection*{H: Descending links}

To prove Proposition~\ref{prop:conn} we will use Bestvina--Brady Morse theory (see \cite{bestvina97}). Given an affine cell complex $Y$, e.g., a simplicial or cube complex, a map $h\colon Y\to \R$ is a \emph{Morse function} if $h$ is affine on cells, non-constant on edges, and discrete on vertices. Given a Morse function $h\colon Y\to\R$ and a cell $C$ in $Y$, $h$ achieves its maximum value on $C$ at a unique vertex, called the \emph{top} of $C$. The \emph{descending link} $\dlk y$ of a vertex $y\in Y^{(0)}$ is the link of $y$ in all the cells with $y$ as their top. The point of Morse theory is that a sufficient understanding of descending links can translate into knowledge about sublevel complexes (see \cite[Corollary~2.6]{bestvina97}).

\begin{proof}[Proof of Proposition~\ref{prop:conn}]
We can extend $f\colon X^{(0)}\to \N$ to a map $f\colon X\to \R$ by extending affinely to each cube, and this is a Morse function. Since $X$ is contractible, it now suffices by \cite[Corollary~2.6]{bestvina97} to prove that for every $x\in X^{(0)}$ with $f(x)>m$, the descending link $\dlk x$ is $(\nu(m+1)-1)$-connected. The descending link of $x$ is the simplicial complex with a $k$-simplex for each $x'=x[1,\altforest]$, where $\altforest$ is an elementary forest with $k+1$ carets and $f(x)$ leaves, with face relation given by removing carets. (For this, it is important that we are using the cubical structure on $X$, not the simplicial structure.) If $f(x)=n$ then this is isomorphic to the matching complex on the graph $L_n$. Here $L_n$ is the graph with vertex set $\{1,\dots,n\}$ and an edge $\{i,i+1\}$ for each $1\le i\le n-1$, and the \emph{matching complex} $\match(\Gamma)$ of a graph $\Gamma$ is the simplicial complex with a simplex for each non-empty finite collection of pairwise disjoint edges of $\Gamma$ with face relation given by inclusion (see Figure~\ref{fig:match}).

\medskip

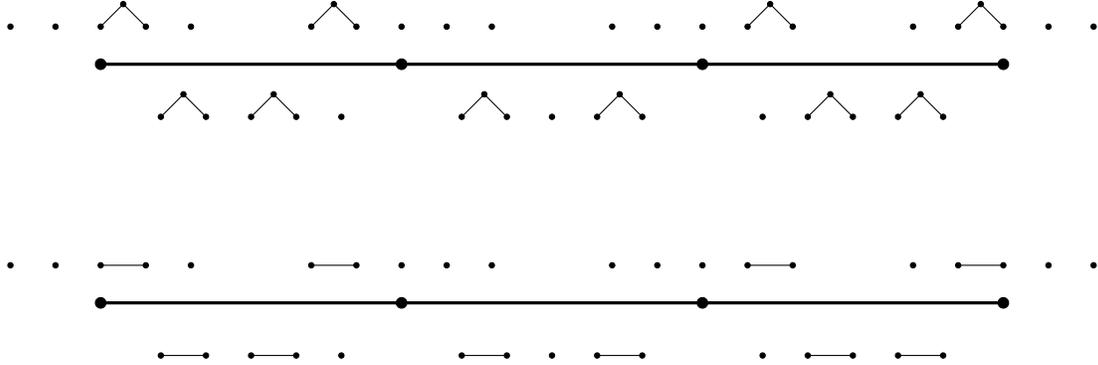
\begin{figure}[hbt]
\begin{tikzpicture}
\draw[line width=1.2] (0,0) -- (12,0);
\filldraw (0,0) circle (2pt)   (4,0) circle (2pt)   (8,0) circle (2pt)   (12,0) circle (2pt);

\draw (0,.5) -- (.3,.8) -- (.6,.5);
\filldraw (-1.2,.5) circle (1pt)   (-.6,.5) circle (1pt)   (0,.5) circle (1pt)   (.6,.5) circle (1pt)   (1.2,.5) circle (1pt)   (.3,.8) circle (1pt);

\draw (4-1.2,.5) -- (4-.9,.8) -- (4-.6,.5);
\filldraw (4-1.2,.5) circle (1pt)   (4-.6,.5) circle (1pt)   (4,.5) circle (1pt)   (4.6,.5) circle (1pt)   (5.2,.5) circle (1pt)   (4-.9,.8) circle (1pt);

\draw (8+.6,.5) -- (8+.9,.8) -- (8+1.2,.5);
\filldraw (8-1.2,.5) circle (1pt)   (8-.6,.5) circle (1pt)   (8,.5) circle (1pt)   (8.6,.5) circle (1pt)   (9.2,.5) circle (1pt)   (8+.9,.8) circle (1pt);

\draw (12-.6,.5) -- (12-.3,.8) -- (12,.5);
\filldraw (12-1.2,.5) circle (1pt)   (12-.6,.5) circle (1pt)   (12,.5) circle (1pt)   (12.6,.5) circle (1pt)   (13.2,.5) circle (1pt)   (12-.3,.8) circle (1pt);

\draw (2,-.7) -- (2.3,-.4) -- (2.6,-.7)   (2-1.2,-.7) -- (2-.9,-.4) -- (2-.6,-.7);
\filldraw (2-1.2,-.7) circle (1pt)   (2-.6,-.7) circle (1pt)   (2,-.7) circle (1pt)   (2.6,-.7) circle (1pt)   (3.2,-.7) circle (1pt)   (2.3,-.4) circle (1pt)   (2-.9,-.4) circle (1pt);

\draw (6.6,-.7) -- (6.9,-.4) -- (6+1.2,-.7)   (6-1.2,-.7) -- (6-.9,-.4) -- (6-.6,-.7);
\filldraw (6-1.2,-.7) circle (1pt)   (6-.6,-.7) circle (1pt)   (6,-.7) circle (1pt)   (6.6,-.7) circle (1pt)   (7.2,-.7) circle (1pt)   (6.9,-.4) circle (1pt)   (6-.9,-.4) circle (1pt);

\draw (10.6,-.7) -- (10.9,-.4) -- (10+1.2,-.7)   (10-.6,-.7) -- (10-.3,-.4) -- (10,-.7);
\filldraw (10-1.2,-.7) circle (1pt)   (10-.6,-.7) circle (1pt)   (10,-.7) circle (1pt)   (10.6,-.7) circle (1pt)   (11.2,-.7) circle (1pt)   (10.9,-.4) circle (1pt)   (10-.3,-.4) circle (1pt);

\begin{scope}[yshift=-1.25in]
\draw[line width=1.2] (0,0) -- (12,0);
\filldraw (0,0) circle (2pt)   (4,0) circle (2pt)   (8,0) circle (2pt)   (12,0) circle (2pt);

\draw (0,.5) -- (.6,.5);
\filldraw (-1.2,.5) circle (1pt)   (-.6,.5) circle (1pt)   (0,.5) circle (1pt)   (.6,.5) circle (1pt)   (1.2,.5) circle (1pt);

\draw (4-1.2,.5) -- (4-.6,.5);
\filldraw (4-1.2,.5) circle (1pt)   (4-.6,.5) circle (1pt)   (4,.5) circle (1pt)   (4.6,.5) circle (1pt)   (5.2,.5) circle (1pt);

\draw (8+.6,.5) -- (8+1.2,.5);
\filldraw (8-1.2,.5) circle (1pt)   (8-.6,.5) circle (1pt)   (8,.5) circle (1pt)   (8.6,.5) circle (1pt)   (9.2,.5) circle (1pt);

\draw (12-.6,.5) -- (12,.5);
\filldraw (12-1.2,.5) circle (1pt)   (12-.6,.5) circle (1pt)   (12,.5) circle (1pt)   (12.6,.5) circle (1pt)   (13.2,.5) circle (1pt);

\draw (2,-.7) -- (2.6,-.7)   (2-1.2,-.7) -- (2-.6,-.7);
\filldraw (2-1.2,-.7) circle (1pt)   (2-.6,-.7) circle (1pt)   (2,-.7) circle (1pt)   (2.6,-.7) circle (1pt)   (3.2,-.7) circle (1pt);

\draw (6.6,-.7) -- (6+1.2,-.7)   (6-1.2,-.7) -- (6-.6,-.7);
\filldraw (6-1.2,-.7) circle (1pt)   (6-.6,-.7) circle (1pt)   (6,-.7) circle (1pt)   (6.6,-.7) circle (1pt)   (7.2,-.7) circle (1pt);

\draw (10.6,-.7) -- (10+1.2,-.7)   (10-.6,-.7) -- (10,-.7);
\filldraw (10-1.2,-.7) circle (1pt)   (10-.6,-.7) circle (1pt)   (10,-.7) circle (1pt)   (10.6,-.7) circle (1pt)   (11.2,-.7) circle (1pt);

\end{scope}
\end{tikzpicture}
\caption{The correspondence between $\dlk x$ with $f(x)=5$ (the top picture, with the forests $\altforest$ representing the simplices) and $\match(L_5)$ (the bottom picture).}
\label{fig:match}
\end{figure}

\medskip

Since $n>m$, it now suffices to show that $\match(L_n)$ is $(\nu(n)-1)$-connected. This is well known (see, e.g., \cite[Proposition~11.16]{kozlov08}), but it is easy to prove so we present a proof here. We will induct on $n$ to prove that this holds, and that moreover $\match(L_n)$ is contractible whenever $n\equiv 2 \mod 3$, and that the inclusion $\match(L_{n-1})\to \match(L_n)$ is a homotopy equivalence whenever $n\equiv 1 \mod 3$. As a base case we can check ``by hand'' that $\match(L_n)$ is non-empty (i.e., $(-1)$-connected) for $n\ge 2$, $\match(L_2)$ is contractible, and $\match(L_3)\to \match(L_4)$ is a homotopy equivalence. Now assume $n\ge 5$. Clearly $\match(L_n)$ is isomorphic to $\match(L_{n-1})$ union the star of $\{n-1,n\}$, and the intersection of $\match(L_{n-1})$ with this star is $\match(L_{n-2})$. Hence $\match(L_n)$ is homotopy equivalent to the mapping cone of the inclusion $\match(L_{n-2})\to \match(L_{n-1})$. If $n\equiv 0,1 \mod 3$ then $\nu(n-1)=\nu(n)$, so $\match(L_{n-1})$ is $(\nu(n)-1)$-connected, and moreover $\match(L_{n-2})$ is $(\nu(n)-2)$-connected, so $\match(L_n)$ is $(\nu(n)-1)$-connected (this follows for example from Van Kampen, Mayer--Vietoris, and Hurewicz). If $n\equiv 2 \mod 3$ then the inclusion $\match(L_{n-2})\to \match(L_{n-1})$ is a homotopy equivalence, so $\match(L_n)$ is contractible. Lastly, if $n\equiv 1 \mod 3$ then $\match(L_{n-2})$ is contractible, so the inclusion $\match(L_{n-1})\to \match(L_n)$ is a homotopy equivalence.
\end{proof}

\bibliographystyle{alpha}
\newcommand{\etalchar}[1]{$^{#1}$}

\end{document}